\documentclass[12]{amsart}
\usepackage{amsmath,amssymb,amsthm,color,enumerate,comment,centernot,enumitem,url,cite}
\usepackage{graphicx,relsize,bm}
\usepackage{mathtools}
\usepackage{array}

\usepackage[backref=page]{hyperref}
\usepackage{hyperref}

\makeatletter
\newcommand{\tpmod}[1]{{\@displayfalse\pmod{#1}}}
\makeatother

\newtheorem{thm}{Theorem}[section]
\newtheorem{lemma}[thm]{Lemma}

\theoremstyle{remark}

\theoremstyle{definition}

\theoremstyle{THM}

\newcommand{\W}{{\mathcal W}}

\newcommand{\FF}{{\mathcal F}}
\newcommand{\Gal}{{\mbox{{\rm{Gal}}}}}

\newcommand{\mmod}[1]{\ \mathrm{mod}\enspace #1}
\newcommand{\Z}{{\mathbb Z}}
\newcommand{\Q}{{\mathbb Q}}

\makeatletter
\@namedef{subjclassname@2020}{%
  \textup{2020} Mathematics Subject Classification}
\makeatother

\def\red#1 {\textcolor{red}{#1 }}
\def\blue#1 {\textcolor{blue}{#1 }}

\numberwithin{equation}{section}

\begin{document}

\title[Abelian Monogenic Trinomials]{A Note on Abelian Monogenic Trinomials}


\author{Lenny Jones}
\address{Professor Emeritus, Department of Mathematics, Shippensburg University, Shippensburg, Pennsylvania 17257, USA}
\email[Lenny~Jones]{doctorlennyjones@gmail.com}

\date{\today}

\begin{abstract}
     An {\em abelian monogenic} polynomial $f(x)\in {\mathbb Z}[x]$ is a monic polynomial of degree $N$ that is irreducible over ${\mathbb Q}$, such that the Galois group of $f(x)$ over $\Q$ is abelian, and   
    $\{1,\theta,\theta^2,\ldots,\theta^{N-1}\}$
    is a basis for the ring of integers of ${\mathbb Q}(\theta)$, where $f(\theta)=0$. In this article, we determine all abelian monogenic trinomials of the form $x^{2n}+ax^{n}+b$, where $n,a,b\in {\mathbb Z}$ with $n\ge 1$ and $ab\ne 0$.  
\end{abstract}


\maketitle
\section{Introduction}\label{Section:Intro}

We say that a monic polynomial $f(x)\in \Z[x]$ is {\em monogenic} if $f(x)$ is irreducible over $\Q$ and  $\{1,\theta,\theta^2,\ldots,\theta^{\deg(f)-1}\}$
  is a basis for the ring of integers, $\Z_K$, of $K={\mathbb Q}(\theta)$, where $f(\theta)=0$. 
  It is well known \cite{Cohen} that 
  \begin{equation} \label{Eq:Dis-Dis}
\Delta(f)=\left[\Z_K:\Z[\theta]\right]^2\Delta(K),
\end{equation} where $\Delta(f)$ and $\Delta(K)$ denote the discriminants over $\Q$, respectively, of $f(x)$ and the number field $K$. We refer to $\left[\Z_K:\Z[\theta]\right]$ as the {\em index} of $f(x)$, and we denote it index($f$). 
Thus, from \eqref{Eq:Dis-Dis}, $f(x)$ is monogenic if and only if $\Delta(f)=\Delta(K)$ or equivalently, index($f$)=1.
 We say that $f(x)$ is {\em abelian monogenic} if $f(x)$ is monogenic and the Galois group of $f(x)$ over $\Q$, denoted $\Gal(f)$, is abelian.  

The investigations in this note were motivated by the following work of Marie-Nicole Gras.
\begin{thm}{\rm \cite{Gras1}}\label{Thm:Gras1}
 Let $N\ge 5$ be an integer with $\gcd(N,6)=1$. There exist only finitely many
abelian extensions of $\Q$ of degree $N$ whose ring of algebraic integers is monogenic.  
\end{thm} 

\begin{thm}{\rm \cite{Gras2}}\label{Thm:Gras2}
Let $\ell\ge 5$ be prime. Then there exists at most one cyclic extension $K/\Q$ of degree $\ell$. Moreover, if $2\ell+1$ is prime, then $K$ is the maximal real subfield of the cyclotomic field of index $2\ell+1$, and if $2\ell+1$ is not prime, then no such cyclic extension exists. 
\end{thm}
\begin{thm}{\rm \cite{Gras3}}\label{Thm:Gras3}
 Let $K/\Q$ be a cyclic extension of degree $\ell^r$, where $\ell\ge 5$ is prime and $r\ge 2$. Let $\Z_K$ denote the ring of integers of $K$.  
 \begin{enumerate}
   \item If $2\ell^r+1$ is not prime, then $\Z_K$ is not monogenic. 
   \item If $2\ell^r+1$ is a prime $p$, and if there exists a prime $q\ne p$ such that $q$ is totally ramified in $K$, then $\Z_K$ is not monogenic.
 \end{enumerate}
\end{thm}
We see that Theorems \ref{Thm:Gras1}, \ref{Thm:Gras2} and \ref{Thm:Gras3} are concerned with abelian monogenic number fields whose degree over $\Q$ is relatively prime to 6. Our focus in this note is on the specific trinomials 
\[\FF_{n,a,b}(x):=x^{2n}+ax^n+b\in \Z[x], \quad \mbox{where} \ n\ge 1 \ \mbox{and} \ ab\ne 0.\] If $\FF_{n,a,b}(x)$ is irreducible over $\Q$, then we make the following two observations: the degree of the number field $\Q(\alpha)$, where $\FF_{n,a,b}(\alpha)=0$, is not relatively prime to 6, and if $d$ is a divisor of $n$ with $d\ge 1$, then $\FF_{d,a,b}(x)$ is also irreducible over $\Q$.  It is our goal to determine all abelian monogenic trinomials of the form $\FF_{n,a,b}(x)$. 

We make some notational remarks. For $\FF_{n,a,b}(x)$, we let 
\[\W:=(a^2-4b)/\gcd(2,a)^2.\]  
 For any integer $N\ge 1$, we let $C_N$ denote the cyclic group of order $N$, and $\Phi_{N}(x)$ denote the cyclotomic polynomial of index $N$ and degree $\phi(N)$.  
For integers $r$ and $m$ with $m\ge 2$, we let the notation ``$r \mmod{m}$" denote the unique integer $z\in \{0,1,2,\ldots,m-1\}$ such that $r\equiv z \pmod{m}$. That is, $r \mmod{m}=z$. For any positive integer $n\ge 2$, we define the {\em radical of $n$}, denoted rad($n$), to be the product of the distinct prime divisors of $n$.

Our main result is the following:
\begin{thm}\label{Thm:Main}
The trinomial $\FF_{n,a,b}(x)$ is abelian monogenic if and only if 
$(n,a,b)$ is contained in one of the following pairwise-disjoint sets: 
\begin{enumerate}
  \item \label{I1} $\{(1,a,b): \W \ \mbox{is squarefree, and}$\\
  \hspace*{.85in} $(a \mmod{4}, \ b\mmod{4})\in \{(0,1),(0,2),(2,2),(2,3)\} \ \mbox{if $2\mid a$}\}$.\\
    In this case, we have $\Gal(\FF_{1,a,b})\simeq C_2$.
   \item \label{I2} $\{(2,\pm 4,2), (2,-5,5)\}$.\\ 
   In this case, we have $\Gal(\FF_{2,a,b})\simeq C_4$.
   \item \label{I3} $\{(2,a,1): \W \ \mbox{is squarefree, and} \ a\mmod{4}\in \{0,3\}\}$.\\
   In this case, we have $\Gal(\FF_{2,a,1})\simeq C_2\times C_2$.
   \item \label{I4} $\{(2^r,-1,1): \ r\ge 2\}$.\\
    In this case, we have $\FF_{2^r,-1,1}(x)=\Phi_{2^{r+1}\cdot 3}(x)$ so that \[\Gal(\FF_{2^r,-1,1})\simeq C_2\times C_2\times C_{2^{r-1}}.\]
  \item \label{I5} $\{(3^s,1,1): \ s\ge 1\}$.\\
    In this case, we have $\FF_{3^s,1,1}(x)=\Phi_{3^{s+1}}(x)$ so that $\Gal(\FF_{3^s,1,1})\simeq C_{2\cdot 3^s}$. 
   \item \label{I6} $\{(3^s,-1,1): \ s\ge 1\}$.\\
    In this case, we have $\FF_{3^s,-1,1}(x)=\Phi_{2\cdot 3^{s+1}}(x)$ so that $\Gal(\FF_{3^s,-1,1})\simeq C_{2\cdot 3^s}$.
  \item \label{I7} $\{(2^r3^s,-1,1): \ rs\ge 1\}$.\\
     In this case, we have $\FF_{2^r3^s,-1,1}(x)=\Phi_{2^{r+1}3^{s+1}}(x)$ so that 
     \[\Gal(\FF_{2^r3^s,-1,1})\simeq C_2\times C_2\times C_{2^{r-1}}\times C_{3^s}.\]
       \end{enumerate}
\end{thm}

\section{Preliminaries}\label{Section:Prelim}

The following lemma shows that there are severe restrictions on the value of $n$ if $\Gal(\FF_{n,a,b})$ is abelian. 
\begin{lemma}\label{Lem:Restrictions on n}
Suppose that $\FF_{n,a,b}(x)$ is irreducible over $\Q$, and that $\Gal(\FF_{n,a,b})$ is abelian. Then $n=2^r3^s$ for some nonnegative integers $r$ and $s$. Furthermore, if $d\ge 1$ is a divisor of $n$, then $\Gal(\FF_{d,a,b})$ is abelian. 
\end{lemma} 
\begin{proof}
Note that if $n=1$, then $\Gal(\FF_{1,a,b})\simeq C_2$ is abelian. So, suppose that $n\ge 2$, and let $p$ be a prime divisor of $n$.
Then, $\FF_{p,a,b}(x)$ is irreducible over $\Q$ since $\FF_{n,a,b}(x)=\FF_{p,a,b}(x^{n/p})$ is irreducible over $\Q$. Observe that if $\alpha$ is a zero of $\FF_{p,a,b}(x)$, then $\alpha\zeta_p$ is also a zero of $\FF_{p,a,b}(x)$, where $\zeta_p$ is a primitive $p$th root of unity. Thus, since $L=\Q(\alpha)$ is normal over $\Q$, $L$ contains the cyclotomic field $\Q(\zeta_p)$. Hence, $[\Q(\zeta_p):\Q]=p-1$ divides $[L:\Q]=2p$, which implies that $p\in \{2,3\}$. 

Suppose next that $d\ge 2$ is a divisor of $n$, and let $K$ and $L$ be the respective splitting fields of $\FF_{n,a,b}(x)$ and the irreducible trinomial $\FF_{d,a,b}(x)$. Then $L\subseteq K$ and $L$ is normal over $\Q$. It follows that $\Gal(\FF_{d,a,b})$ is abelian since otherwise $\Gal(\FF_{n,a,b})$ would have a nonabelian quotient, which is impossible. 
\end{proof} 

The following four theorems describe the abelian monogenic trinomials $\FF_{n,a,b}(x)$ for, respectively, $n\in \{1,2,3,6\}$. The proofs can be found in \cite{Jones degree 12}.
\begin{thm}\label{Thm:Main d=1} Suppose that $\FF_{1,a,b}(x)$ is irreducible over $\Q$. Then $\FF_{1,a,b}(x)$ is abelian monogenic if and only if $\W$ is squarefree and 
 \[(a \mmod{4}, \ b\mmod{4})\in \{(0,1),(0,2),(2,2),(2,3)\} \quad \mbox{if $2\mid a$.}\]
\end{thm}

\begin{thm}\label{Thm:Main d=2} Suppose that $\FF_{2,a,b}(x)$ is irreducible over $\Q$. Then $\FF_{2,a,b}(x)$ is abelian monogenic with $\Gal(\FF_{2,a,b})\simeq$
\begin{enumerate}
   \item \label{g4:1} $C_4$ if and only if $\FF_{2,a,b}(x)\in \{x^4\pm 4x^2+2,\ x^4-5x^2+5\}$; 
   \item \label{g4:2} $C_2\times C_2$  if and only if $\W$ is squarefree and  
   \[\FF_{2,a,b}(x)\in \{x^4+ax^2+1: a\mmod{4}\in \{0,3\}\}.\] 
   \end{enumerate}
\end{thm}

\begin{thm}\label{Thm:Main d=3}  
Suppose that $\FF_{3,a,b}(x)$ is irreducible over $\Q$. Then
$\FF_{3,a,b}(x)$ is abelian monogenic if and only if $\FF_{3,a,b}(x)\in \{x^6-x^3+1,\ x^6+x^3+1\}$, in which case,  $\Gal(\FF_{3,a,b})\simeq C_6$. 
\end{thm}

\begin{thm}\label{Thm:Main d=6}  
Suppose that $\FF_{6,a,b}(x)$ is irreducible over $\Q$. Then
$\FF_{6,a,b}(x)$ is abelian monogenic if and only if $\FF_{6,a,b}(x)=x^{12}-x^6+1$, in which case, $\Gal(\FF_{6,a,b})\simeq C_2\times C_6$. 
\end{thm}

The next lemma is an application to our specific situation of a theorem due to Kaur, Kumar and Remete \cite[Theorem 1.1]{KKR}.  
\begin{lemma}\label{Lem:KKR}
    Let $n,a,b,r,s\in \Z$, such that $n=2^r3^s$ with $n\not \in \{1,2,3,6\}$ and $ab\ne 0$. Let $k=n/{\rm rad}(n)$. Then $k>1$ and  
    \[\FF_{n,a,b}(x)=\FF_{{\rm rad}(n),a,b}(x^k).\] If $\FF_{n,a,b}(x)$ is irreducible over $\Q$, then 
    $\FF_{n,a,b}(x)$ is monogenic if and only if all of the following conditions are satisfied: 
   \begin{enumerate}
   \item \label{KKR:I1} $b$ is squarefree,
    \item \label{KKR:I2} $q$ does not divide the {\rm index($\FF_{n,a,b}$)} for all prime divisors $q$ of $k$,
     \item \label{KKR:I3} $\FF_{{\rm rad}(n),a,b}(x)$ is monogenic. 
   \end{enumerate}
 \end{lemma}

\section{Proof of Theorem \ref{Thm:Main}}
\begin{proof}
Suppose first that $\FF_{n,a,b}(x)$ is abelian monogenic. Thus, $n=2^r3^s$, for some nonnegative integers $r$ and $s$, by Lemma \ref{Lem:Restrictions on n}. Then, utilizing Theorems \ref{Thm:Main d=1}, \ref{Thm:Main d=2}, \ref{Thm:Main d=3} and \ref{Thm:Main d=6}, along with Lemma \ref{Lem:KKR}, and considering the four cases
\begin{equation}\label{cases}
(r=0,s=0), \ \ (r \ge 1,s=0), \ \ (r=0,s \ge 1) \ \ \mbox{and} \ \ (r \ge 1,s \ge 1),
\end{equation} we arrive at the exact sets given in the statement of the theorem. 

Since the methods for each of the cases in \eqref{cases} are similar, we provide details only for the case $(r\ge 1, s=0)$, which will yield the sets in items \eqref{I2}, \eqref{I3} and \eqref{I4} in the statement of the theorem. 
If $r=1$, then $\FF_{2,a,b}(x)$ is monogenic by assumption. If $r\ge 2$, then, since 
\[\FF_{n,a,b}(x)=\FF_{2^r,a,b}(x)=\FF_{2,a,b}\left(x^{2^{r-1}}\right)\] is monogenic, it follows from Lemma \ref{Lem:KKR} that $\FF_{2,a,b}(x)$ is monogenic. That is, in either situation, we have that $\FF_{2,a,b}(x)$ is monogenic. Therefore, we deduce from Theorem \ref{Thm:Main d=2} the two possibilities for $(a,b)$: 
\begin{equation}\label{Poss}
(a,b)\in \{(\pm 4,2),(-5,2)\} \quad \mbox{or} \quad a\mmod{4}\in \{0,3\},\ b=1 \ \mbox{with}\ \W \ \mbox{squarefree}.
\end{equation} 
Hence, if $r=1$, then $(n,a,b)=(2,a,b)$ is contained in either the set of item \eqref{I2} or the set of item \eqref{I3}. If $r\ge 2$, then it is easy to verify, using a computer algebra system, that $\FF_{2^2,\pm 4,2}(x)$ and $\FF_{2^2,-5,2}(x)$ are nonabelian. Consequently, $\FF_{2^r,\pm 4,2}(x)$ and $\FF_{2^r,-5,2}(x)$ are nonabelian for any $r\ge 2$ by Lemma \ref{Lem:Restrictions on n}. For the second possibility in \eqref{Poss}, it follows from \cite{JonesMEOPGG} that $\FF_{2^2,a,1}(x)$ is abelian only when $a=-1$. Thus, in this situation we get $\FF_{2^2,-1,1}(x)=\Phi_{24}(x)$, so that $\FF_{2^2,-1,1}(x)$ is monogenic and $\Gal(\FF_{2^2,-1,1})\simeq C_2\times C_2\times C_2$. Therefore, $(n,a,b)=(2^r,-1,1)$ is contained in the set of item \eqref{I4} in the theorem since, when $r\ge 3$, we have $\FF_{2^r,-1,1}(x)=\Phi_{2^{r+1}\cdot 3}(x)$.

 Conversely, if $(n,a,b)$ is contained in one of the seven sets given in the statement of the theorem, then it is straightforward to confirm that $\FF_{n,a,b}(x)$ is abelian monogenic. For example, 
 if $(n,a,b)\in \{(2^r3^s,-1,1): \ rs\ge 1\}$, then it is easy to verify that  
 $\FF_{2^r3^s,-1,1}(x)=\Phi_{2^{r+1}3^{s+1}}(x)$. Hence,  
 \[\Gal(\FF_{2^r3^s,-1,1})\simeq C_2\times C_2\times C_{2^{r-1}}\times C_{3^s},\] and $\FF_{2^r3^s,-1,1}(x)$ is monogenic by \cite[Theorem 2.6]{Washington}. 
\end{proof}








\end{document}